\documentclass[12pt]{amsart}

\usepackage{graphicx}
\usepackage{eucal}
\usepackage{xcolor}
\usepackage{amsfonts,amssymb,amscd}
\usepackage[all,knot]{xy}
\usepackage{tikz}
\usepackage{braids}

\newtheorem{theorem}{Theorem}[section]
\newtheorem{lemma}[theorem]{Lemma}
\theoremstyle{definition}
\newtheorem{definition}[theorem]{Definition}

\newtheorem{proposition}[theorem]{Proposition}
\newtheorem{corollary}[theorem]{Corollary}
\newtheorem{observation}[theorem]{Observation}
\newtheorem{conjecture}[theorem]{Conjecture}
\newtheorem{problem}[theorem]{Problem}

\theoremstyle{remark}

\numberwithin{equation}{section}
\newcommand{\Rib}{\operatorname{Rib}}
\newcommand{\Rop}{\operatorname{Rop}}
\usepackage{graphicx}
\usepackage{hyperref}
\usepackage{geometry}
\usepackage[all]{xy}

\begin{document}
\title{Profiles of Critical Flat Ribbon Knots}
\author{Jos\'{e} Ayala}
\address{Universidad de Tarapac\'a, Iquique, Chile} 
\maketitle
\baselineskip=20 true pt
\maketitle \baselineskip=1.1\normalbaselineskip

\begin{abstract} \maketitle \baselineskip=1.2\normalbaselineskip 
 
 The main open problem in geometric knot theory is to provide a tabulation of knots based on an energy criterion, with the goal of presenting this tabulation in terms of global energy minimisers within isotopy classes, often referred to as ideal knots. Recently, the first examples of minimal length diagrams and their corresponding length values have been determined by Ayala, Kirszenblat, and Rubinstein. This article is motivated by the scarcity of examples despite several decades of intense research. Here, we compute the minimal ribbonlength for some well-known knot diagrams, including the Salomon knot and the Turk’s head knot. We also determine the minimal ribbonlength for the granny knot and square knot using a direct method. We conclude by providing the ribbonlength for infinite classes of critical ribbon knots, along with conjectures aimed at relating ribbonlength to knot invariants in pursuit of a metric classification of knots.
 \end{abstract}

\section{Introduction} \label{intro}

What is the minimum length required for a strip of constant width to form a knot diagram? This question was asked by Kauffman \cite{kauffman}, wherein he proposed a model involving a folded flat ribbon characterised by a piecewise linear core. He asked for the least ribbonlength (the ratio of length to width) needed for the trefoil and figure-eight knot. Despite the efforts, these remain open; in fact, the unknot is the only knot whose ribbonlength has been explicitly computed. Other fundamental questions about folded ribbons remain open, such as the existence of minimisers. Is minimal folded ribbonlength achieved by the minimal crossing number? For an overview of folded ribbon knots, see \cite{denne1}.

A flat ribbon knot is the embedding of an annulus into the Euclidean plane, where its core corresponds to an immersed knot diagram. In \cite{ribbon1}, we solved the ribbonlength problem for certain immersed flat ribbon knots and links. We also proved the existence of minimisers. Therein, we computed the minimal ribbonlength and obtained the ideal shape of the Hopf link and certain infinite families of link diagrams. So far, the only immersed flat ribbon knot diagrams of known minimal ribbonlength are the standard unknot and the trefoil.

The 3-dimensional counterpart of the ribbonlength problem, the ropelength problem, aims to find the least ratio of length to thickness required to tie a given knot. Several models have been proposed. Again, the unknot, which is minimised by the round circle, is the only knot whose ropelength has been explicitly computed under all known models. For links, certain planar chains are minimised by linked piecewise $C^2$ loops \cite{sullivan1}. Despite the extensive theory, examples remain scarce. One recent example is provided in \cite{kusner1}. This lack of examples is one of the motivations for this paper.

To partially solve the ribbonlength problem, we have embedded the space of ribbon knots and links into a larger, more manageable space of disk diagrams. The key point is that when length minimisers in the disk diagram space conform to the ribbon type, they solve the ribbonlength problem—though this may not always be the case.

  In this article, we establish a simple variational criterion to check local criticality (see Lemma \ref{lemcomp}). We also find the exact minimal ribbonlength values for the standard diagrams of the Salomon, Granny, Square, Turk's head, and the connected sums of Salomon knots. We note that the square lattice, together with Lemma \ref{lemcomp}, provides a suitable framework for verifying the existence and computing the ribbonlength of certain critical diagrams. The strategy consists of starting with a suitable initial diagram so that a gradient descent leads to a realisation that can be verified to be ribbonlength critical, and sometimes even minimal. In Figure \ref{fig:sal1}, we show the Salomon link and its corresponding ribbonlength minimiser diagram obtained in Theorem \ref{sal1}. In the second row of Figure \ref{fig:sal1}, we show the second element in one of the families considered here, whose ribbonlength is computed in Theorem \ref{door}. For the convenience of the reader, in Section \ref{ribbons}, we provide a brief review of basic definitions from \cite{ribbon1}. In Section \ref{minrib}, we prove minimality for an infinite family of diagrams called the rectangular family (Theorem \ref{rect}). In Theorem \ref{door}, we establish criticality for another infinite family. We conclude by presenting conjectures aimed at relating ribbonlength to knot invariants in pursuit of a metric classification of knots.

\begin{figure} [h!]
\centering
\includegraphics[width=.6\textwidth]{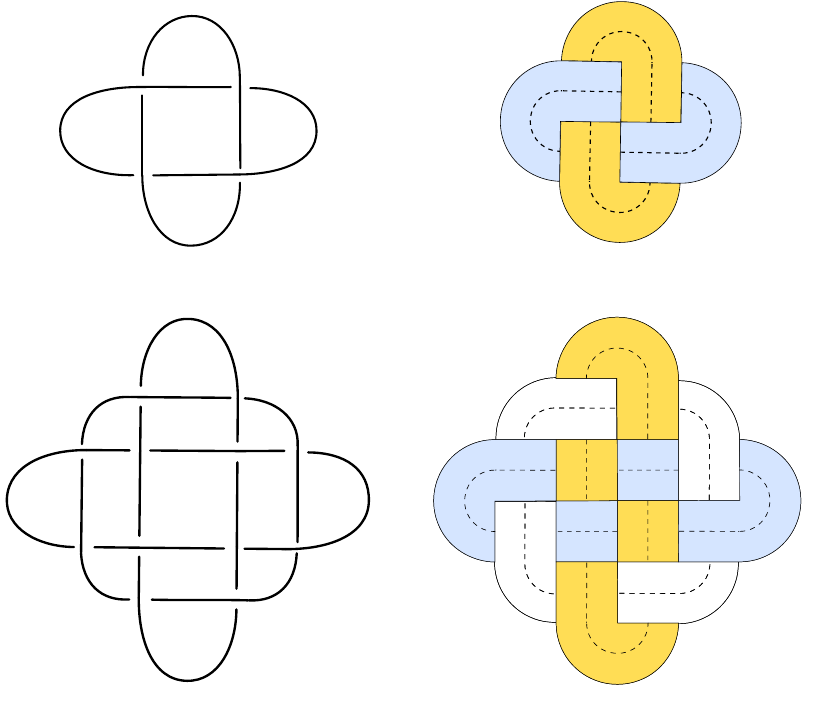}
\caption{Above: The Salomon knot has a minimal ribbonlength of \(2\pi + 8\). The dashed trace represents the core of the ribbon. Below: The second element of an infinite family of diagrams described in Section \ref{infinite}.}
\label{fig:sal1}
\end{figure}

The diagrams considered here are related to several planar representations of knots, which have been used across various interconnected contexts. For instance, Lomonaco and Kauffman presented a system of tiles to define quantum knot system \cite{kauffman2}. Also, Grosberg and Nechaev \cite{grosberg}, and Nechaev \cite{nechaev}, studied lattice knots and their interrelations with the Potts model in statistical mechanics. Additionally, Even-Zohar, Hass, Linial, and Nowik \cite{chaim} showed the existence of certain universal lattice diagrams that can represent all knots. Bounding ribbonlength may benefit from these previous models.

\section{Spaces of Immersed Flat Ribbon Knots and Links} \label{ribbons}

To develop a theory for immersed flat ribbon knots, we had to address several technicalities. For the convenience of the reader, we informally discuss the main definitions and some of their important consequences. For precise definitions and statements, please refer to \cite{ribbon1}.

The ribbons are modelled as the tubular neighbourhood of and immersed $C^1$ planar loop $\gamma$. These are given by $\Gamma :S^1 \times (-1,1) \to \mathbb R^2$ when $\Gamma(s,t)=\gamma(s)+tN(s)$, where $N(s)$ is a continuous choice of normal vectors to $\gamma$ so that $\Gamma(s,0)=\gamma(s)$, for $s\in S^1$ is satisfied. Here $\Gamma$ will be interpreted as the immersion of an open annulus.

The {\bf separation bound} is about maintaining apart from itself the space around the core $\gamma$ of the ribbon. It also ensures that the curve does not loop back too closely. This condition keeps the ribbon well-behaved (see \cite{ribbon1}).

The {\bf crossing condition} establish that $\gamma$ only has well behaved crossings, these being double points or double intervals. Near each crossing, the curve behaves as expected. This definition is made to avoid typical pathological behaviour of immersions and also to obtain Theorem \ref{doublepoint} which gives an important finiteness result.

A $C^1$ immersed loop $\gamma: S^1\to\mathbb R^2$ satisfying the separation bound for a collection of closed intervals covering $S^1$ and the crossing condition is called a {\bf ribbon loop}. The immersed open annular neighbourhood $\Gamma(S^1 \times (-1,1))$ of a ribbon loop $\gamma(S^1)$ is called the ribbon associated to the loop. A finite collection of $C^1$ loops is called a ribbon link. The space of these objects has the standard $C^1$ topology. 

A $C^1$ immersed planar ribbon loop is said to satisfy the {\bf double point condition} if it has finitely many isolated double points or double intervals.

\begin{theorem}\label{doublepoint} Ribbon loops and links satisfy the double point condition. 
\end{theorem}

\begin{proposition} \label{disk} Each connected complementary region of a ribbon loop or link contains an open disk with radius 1.
\end{proposition}

With these two results about ribbon loops we soon introduce a bigger space that contains the space of all ribbons. In this space, called disk space, we look for length minimisers provided they satisfy the conditions to be ribbon diagrams.

The {\bf weak separation bound} means that every point on the curve has two circles touching the curve lying on opposite sides of the curve i.e the curve is not bending so tightly. This condition keeps the curve from having sharp bends or being too ``thin" ensuring that the curve has smoothness and thickness throughout.

The  {\bf disk space} $\mathcal D$ is the space of finite collections of $C^1$ immersed planar loops satisfying the {double point condition}, {the weak separation bound}, and its elements are such that each complementary region contains an open disk with radius 1. $\mathcal D$ has the standard $C^1$ topology. Note that a ribbon loop is an element is disk space $\mathcal D$ that satisfy the separation bound.

A {\bf ribbon knot diagram} is a finite collection of ribbon loops with under/over information as its crossings.

\begin{theorem}\label{minimalribbon}
Suppose that $\gamma$ is a ribbon knot diagram. Then minimal length ribbon representatives of $\gamma$ exist.
\end{theorem}

A  diagram is called {\bf cs} if its core is a finite $C^1$ concatenation of arcs of circles of radius 1 and line segments. 

\begin{theorem}\label{ribcs} Length minimisers in each component of disk space exist and are $cs$ curves. Suppose that a length minimiser of a ribbon knot or link in disk space is ribbon. Then this disk space minimiser is also a length minimiser in ribbon space
\end{theorem}

{\bf Observation.} Note that our focus is on finding minimal length diagrams for standard diagrams of knots within the same combinatorial type. This approach contrasts with searching across all possible realisations of a knot.

\section{Criticality for Ribbonlength and Connections to Ropelength}

To quantify geometric complexity under a fixed width, we introduce the ribbonlength functional as a planar analogue of ropelength. Just as ropelength normalises the total length of a spatial curve by the radius of its enclosing tube, ribbonlength normalises the length of a planar diagram by the fixed width of the ribbon. This yields a scale-invariant measure of complexity of a diagram.

\begin{definition} The ribbonlength $\mbox{Rib}(\gamma)$ of a ribbon knot or link $\gamma$ is the ratio of the length of the core $\gamma$ to the width of the ribbon $\Gamma$,
$$\mbox{Rib}(\gamma)= \frac{\mbox{Length}(\gamma)}{\mbox{Width}(\Gamma)}= \frac{\mbox{Length}(\gamma)}{2}.$$
\end{definition}

A flat ribbon knot diagram is said to be {\bf{critical for ribbonlength}} if the directional derivative of the length functional $\ell({\bf x})$ vanishes for all directions ${\bf v}$ that preserve the ribbon constraint to first order
\[
D\ell({\bf v}) = \left.\frac{d}{dt} \ell({\bf x} + t{\bf v})\right|_{t=0} = 0.
\]
Here, ${\bf x} = (x_1, \ldots, x_{2n})$ denotes the centres of the $n$ disks assigned to the bounded planar regions of the diagram, and $D\ell({\bf v}) = \sum \hat{\bf u}_j \cdot {\bf v}$ represents the sum of directional contributions from the outward unit tangents $\hat{\bf u}_j$ at the arc–segment junctions of the moving disk. For a full derivation, see Section~5 in~\cite{ribbon1}. 

Explicitly, if only the $i$th disk moves along a smooth curve $t \mapsto {\bf c}(t)$ with velocity ${\bf v} = \frac{d{\bf c}}{dt}$ and all other disks are fixed, then:
\[
D\ell({\bf v}) = \lim_{t \to 0^+} \left[\ell(x_1, \ldots, x_{2i - 1} + c_1(t), x_{2i} + c_2(t), \ldots, x_{2n}) - \ell({\bf x})\right] = \sum_{j=1}^{2k} \hat{\bf u}_j \cdot {\bf v},
\]
where $2k$ denotes the number of (possibly degenerate) segments tangent to the $i$th disk. 

\subsection{Physical realisation of criticality.} A ribbonlength-critical diagram satisfies two intrinsically linked conditions: it is a $cs$ curve with balanced geometric tension at each vertex, and its associated ribbon $\Gamma$ locally minimises length while respecting the separation bound. The geometric balance arises as a consequence of our variational principle, unifying both aspects under a single physical mechanism.

\subsection{Connections to ropelength theory}
Our variational framework for ribbonlength closely mirrors the ropelength theory for thick knots in $\mathbb{R}^3$, both conceptually and technically as follows.

\noindent
 \textbf{Constraint-driven minimisation.} Both models prevent collapse through enforced constraints that guarantee geometric non-degeneracy. In ropelength, this is achieved via the normal injectivity radius, ensuring a minimum distance between strands of a spatial curve. In our ribbonlength model, the analogous role is played by the local separation bound. This condition ensures that the curve cannot collapse too closely to itself and corresponds to a local medial axis control.
 
 \noindent
\textbf{Tension balance and equilibrium.} In both settings, critical configurations represent equilibria under length minimisation constrained by thickness or width. For ropelength, this yields curvature-strut balance equations derived from first variation analysis in the space of embedded thick curves  \cite{sullivan2}. In the ribbonlength setting, the first variation is computed in the space $\mathcal{D}$ of disk diagrams, leading to a vertex tension balance condition at arc-segment junctions \cite{ribbon1}. Circular arcs with unit curvature play the role of constrained geometry, mirroring the bounded curvature segments in thick knot models.

\noindent
 \textbf{Length complexity bounds.} Both theories relate length to topological complexity, using length to control invariants such as crossing number. In ropelength, total length bounds quantities such as crossing number and bridge number \cite{sullivan1}. Our Theorem~7.1 in \cite{ribbon1} shows a comparable result in the planar context: the number of crossings in ribbon diagram is bounded in terms of the ribbonlength. This reflects a shared philosophy of using geometric cost to regulate diagrammatic or spatial complexity.

 \subsection{Gradient Descent}

In Section 5 of \cite{ribbon1}, we introduced a geometric gradient descent method to reduce the total length of \( cs \) curves. The goal is to search for critical configurations within the space \( \mathcal{D} \) of disk diagrams, while ensuring that the constraints are preserved throughout the deformation.

We model the diagram as a mechanical system, the disks represent rigid units, and the curve enclosing them acts like an elastic band. The descent process involves computing the first variation of the length functional as the centres of the disks are varied. Each \( cs \) segment of the diagram is decomposed into a sequence of geometric components, an arc of a unit circle, followed by a straight segment, and ending with another arc. For each such element, we compute a vector contribution \( \mathbf{u}_i \) associated to the transition points between circular and linear segments.

For example, in Figure~\ref{fig:gradsal}, the contributions acting on the disk \( D_1 \) sum to \( \mathbf{u}_1 + \mathbf{u}_2 \). These contributions define a local force vector acting on the centre of the disk. To descend in length, we perturb the position of the disk centre along a smooth path whose initial tangent is a vector \( \mathbf{v} \). If this perturbation preserves the ribbon constraint and the inner product \( \langle \mathbf{u}_1 + \mathbf{u}_2, \mathbf{v} \rangle \) is positive, then the motion produces a decrease in length.

This vector based framework provides a concrete geometric interpretation of the gradient flow on the configuration space of disks. The reader may find it helpful to keep this picture in mind while following the upcoming formal proofs. For precise definitions and detailed context, we refer the reader to \cite{ribbon1}.

\section{Computing Minimal Ribbonlength}\label{minrib}

\begin{lemma} \label{lemcomp}
Let $A, B, C, D$ be unit-radius disks with centres $a = (0, 0)$, $b = (2, 0)$, $c = (0, 2)$, and $d = (-2, 0)$. Consider the $scs$ curve of length $4 + \pi$ with fixed terminal points $P = (-1, 0)$ and $Q = (1, 0)$, and directions pointing orthogonally downward, as shown in Figure \ref{fig:comp}. Then,
\begin{enumerate}
    \item when rolling $C$ around $A$ while keeping the centres of all circles at least a distance of 2 apart, the $scs$ curve is locally length-minimal and satisfies the separation bound.
    \item when rolling $C$ around $A$ while keeping the centres of all circles at least a distance of 2 apart, the family of admissible perturbed $cs$ curves is monotonically length-increasing, achieving a maximum of $\frac{5\pi}{3} + 2$ at which the separation bound is not satisfied.
    \item under these $cs$ perturbations, there are two saddle points of length.
\end{enumerate}
\end{lemma}

\begin{proof} 
(1) Let the contribution to the negative gradient of length be denoted by ${\bf u} = {\bf u_1} + {\bf u_2}$. It is not possible to perturb $C$ in every direction ${\bf v}$ of the negative gradient since $A$ and $C$ cannot overlap. Note that the first allowed perturbation ${\bf v}$ to $C$, being length non-increasing, satisfies ${\bf u} \cdot {\bf v} = 0$. We conclude criticality by noting that the directional derivative being zero indicates that the $scs$ curve is locally a critical point of length in the direction of ${\bf v}$. The length of this $scs$ curve is $4 + \pi \approx 7.141$. Since the line segments are parallel and distant apart 2 units, this curve satisfies the separation bound.

\begin{figure} [h!]
\centering
\includegraphics[width=1\textwidth]{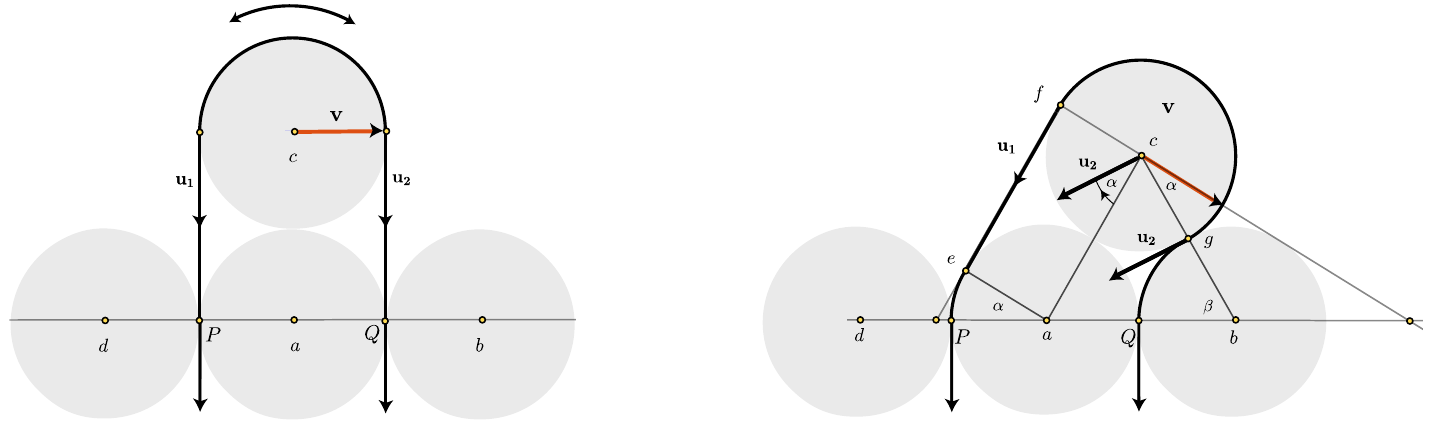}
\caption{We start with an $scs$ elastic-like curve and roll $C$ over $A$ clockwise to achieve the limit curve on the right. The left figure illustrates that the first variation vanishes since ${\bf u} \cdot {\bf v} = 0$.}
\label{fig:comp}
\end{figure}

(2) The vector ${\bf u_1}$ in the $csc$ curve is determined by a common tangent between $A$ and $C$, and ${\bf u_2}$ is determined by a common tangent between $C$ and $B$. Since $C$ rolls over $A$, say clockwise, it is not hard to see that the angle of the vector ${\bf u_2}$ decreases monotonically clockwise from $3\pi/2$ to $3\pi/2 - \pi/3 = 4\pi/3$, and therefore ${\bf u_2} \cdot {\bf v}$ is monotonically increasing during the rolling since ${\bf v}$ remains fixed, at the center of $C$, see Figure \ref{fig:comp}. Additionally, by construction, ${\bf u} \cdot {\bf v} = 0$. By bilinearity, we have
$${({\bf u_1} + {\bf u_2}) \cdot {\bf v} = {\bf u_1} \cdot {\bf v} + {\bf u_2} \cdot {\bf v} = {\bf u_2} \cdot {\bf v} \geq 0.}$$

Since the centres of all circles remain at least a distance of 2 apart, the rolling must stop when $\|c - b\| = 2$. In this limit case, $\alpha = \pi/3$ and $\beta = \pi/6$. Therefore, the limit $cs$ curve is a local maximum of length $\pi/3 + 2 + \pi + \pi/3 + \pi/6 = \frac{11\pi}{6} + 2 \approx 7.756$. Note that this curve does not satisfy the separation bound since the segment ${ef}$ and the arc ${Qg}$ are less than 2 units apart. Note this is not a variational critical point, as the extremum arises from the constraint of keeping some disks fixed.

(3) By a direct application of the mountain pass lemma, we conclude that there are two saddle points of length: one by rolling $C$ clockwise and the other by rolling $C$ counterclockwise from the initial critical position. 
\end{proof}

\begin{figure} [h!]
\centering
\includegraphics[width=.8\textwidth]{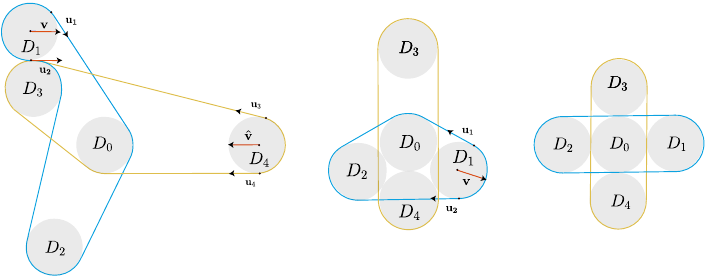}
\caption{Examples of Salomon knot disk diagrams. Left: there are two scenarios under the descent. The disk $D_4$ is free to approach $D_0$, while disk $D_3$ blocks $D_1$. Right: a minimal-length Salomon disk diagram. This minimiser satisfies the separation bound.}
\label{fig:gradsal}
\end{figure}

\begin{theorem} \label{sal1}
The minimal ribbonlength for the standard Salomon knot diagram is $8 + 2\pi$.
\end{theorem}

\begin{proof}  
Because length minimisers in disk space are $cs$ curves (Theorem \ref{ribcs}), we can consider a $cs$ representative for the diagram of the Salomon knot (see Figure \ref{fig:gradsal}). Without loss of generality, fix the position of $D_0$. We first prove that none of the disks $D_1, D_2, D_3, D_4$ prevent each other from intersecting $D_0$. Suppose that $D_3$ is preventing $D_1$ from touching $D_0$. Note that it is not always possible to perturb disk $D_1$ in every direction of the negative gradient, because otherwise it would intersect the interior of disk $D_3$. However, we can perturb disk $D_1$ in the direction of a vector $\bf{v}$, which makes an acute angle with the negative gradient, in this case $\bf{u}_1 + \bf{u}_2$, in order to decrease the length of the curve. Therefore, since $(\bf{u}_1 + \bf{u}_2) \cdot \bf{v} > 0$, this configuration is not length-minimising.
 
We assert that each of the disks $D_1, D_2, D_3, D_4$ intersects $D_0$ at a single point. Without loss of generality, suppose that $D_4$ does not touch $D_0$. Let the contribution to the negative gradient of length from the centre of $D_4$ be denoted by $\bf{u}_3 + \bf{u}_4$ (see Figure \ref{fig:gradsal}). If $D_0$ and $D_4$ do not intersect, we may decrease the length of the disk diagram by perturbing disk $D_4$ in the direction of the negative gradient $\hat{\bf v}$. Note that $D_0$ and $D_4$ belong to different regions separated by the disk diagram; therefore, the distance between their centres is at least 2, i.e., no overlapping is possible. We conclude that $D_0$ and each of the disks $D_1, D_2, D_3, D_4$ intersect at a single point (see Figure \ref{fig:sal1}, right). The final configuration attains all lower bounds, since the minimal perimeter enclosing 3 disks is given by a stadium curve of length $8 + 2\pi$. We obtain a minimum-length element in disk space $\mathcal{D}$ of length $16 + 4\pi$ that satisfies the ribbon conditions. We conclude that the minimal ribbonlength of the standard Salomon diagram is $8 + 2\pi$.
\end{proof}

\begin{figure} [h!]
\centering
\includegraphics[width=1\textwidth]{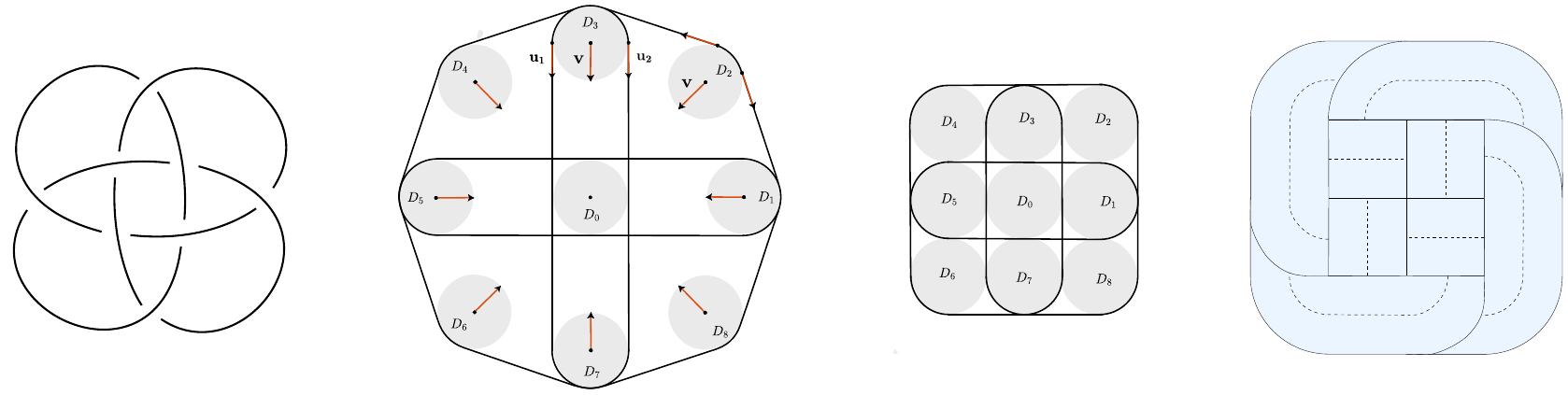}
\caption{}
\label{fig:turk}
\end{figure}

\begin{theorem} \label{turk}
The minimal ribbonlength for the standard Turk's head diagram is $16 + 3\pi$.
\end{theorem}

\begin{proof} 
Since length minimisers in disk space are $cs$ curves (see Theorem \ref{ribcs}), we consider a suitable $cs$ representative for the Turk's head diagram to facilitate the application of the gradient descent. Let $D_i$ be unit-radius disks with centres $d_i$, $i = 0, \ldots, 8$, with one disk in each bounded region separated by the diagram. Let $D_0$ be fixed with $d_0 = (0, 0)$, and let $d_1 = (4, 0)$, $d_2 = (3, 3)$, $d_3 = (0, 4)$, $d_4 = (-3, 3)$, $d_5 = (-4, 0)$, $d_6 = (-3, -3)$, $d_7 = (0, -4)$, and $d_8 = (3, -3)$. The second illustration in Figure \ref{fig:turk} shows the chosen initial $cs$ diagram. Note that a similar analysis as in Theorem \ref{sal1} can be applied to show that none of the peripheral disks $D_1, D_3, D_5, D_7$ can be prevented from touching $D_0$.

First, consider linear perturbations $p(t) = d_0 t + d_i (1 - t)$, $t \in [0, 1/3]$, to the centres of $D_i$ for $i = 2, 4, 6, 8$ while leaving the other parts of the diagram fixed. For example, when considering $p(t) = d_0 t + d_2 (1 - t)$, the gradient descent checks length non-increasing for the perturbed part of the diagram for $t \in [0, 1/3]$ since ${\bf v} \cdot ({\bf u_1 + u_2}) \geq 0$. Additionally, $p(1/3) = (2, 2)$, which also guarantees the separation bound. The cases when $d_4, d_6, d_8$ are identical.

Next, consider linear perturbations $p(t) = d_0 t + d_i (1 - t)$, $t \in [0, 1/2]$, to the centres of $D_i$ for $i = 1, 3, 5, 7$. When $p(t) = d_0 t + d_1 (1 - t)$, the gradient descent checks length non-increasing for the perturbed part of the diagram for $t \in [0, 1/2]$ since ${\bf v} \cdot ({\bf u_1 + u_2}) \geq 0$. Additionally, $p(1/2) = (2, 0)$, which also guarantees the separation bound. The cases when $d_3, d_5, d_7$ are identical. The obtained realisation is a disk diagram satisfying the separation bound, and therefore is of ribbon type (see Figure \ref{fig:turk}). 

Note we can decompose the disk diagram in Figure \ref{fig:turk} into two loops: one being a Salomon disk diagram, and $\gamma_2$, being the perimeter of 9 disks with centres in the square lattice. In Theorem \ref{sal1}, we proved the minimality of $\gamma_1$ whose length is $16 + 4\pi$. By a variational argument similar to those given above, $\gamma_2$ is also minimal with length $16 + 2\pi$. Since the configuration attains both bounds for minimality, it gives a global minimum for the standard Turk's head knot diagram, whose ribbonlength is the sum of both loop lengths divided by 2, yielding a total of $16 + 3\pi$.
\end{proof} 

\begin{corollary} \label{consal}
The minimal ribbonlength of the connected sum of $n$ Salomon links is $8n + 2\pi$.
\end{corollary}

\begin{proof} 
In Theorem \ref{sal1}, we established that the minimal ribbonlength for the Salomon knot is $8 + 2\pi$. We can perform a connected sum of two Salomon knots by cutting each along one of its four creases. This operation would subtract $\pi$ from each of the length components, resulting in a total subtraction of $2\pi$ from the ribbonlength. For the connected sum of three Salomon knots, we must subtract a total of $2 \times 2\pi$ from the total ribbonlength. A simple inductive process leads to the conclusion that the ribbonlength of the connected sum of $n$ of these diagrams is $(8 + 2\pi)n - 2\pi(n - 1) = 8n + 2\pi$, as desired.
\end{proof}

The following result is a direct application of the ribbonlength of the trefoil established in \cite{ribbon1}. This approach allows us to avoid using gradient descent. Despite differing in chirality, the knots in question have the same minimal ribbonlength.

\begin{corollary}
The granny knot and the square knot have identical minimal ribbonlengths for their respective standard diagrams, each measuring $12 + 3\pi$.
\end{corollary}

\begin{proof}
Both the granny and the square knot are connected sums of two trefoil knots. Apart from their crossings, these are diagrammatically identical. The former is a connected sum of two right-handed trefoil knots, while the latter is a connected sum of a right-handed and a left-handed trefoil knot. Since $Rib(\text{trefoil}) = 6 + 2\pi$ (see \cite{ribbon1}), each trefoil contributes $6 + 2\pi$ to the total ribbonlength. However, in the connected sum operation, one unit disk perimeter overlaps between the two trefoils. We conclude that $\text{Rib}(\text{granny}) = \text{Rib}(\text{square}) = (6 + 2\pi) + (6 + 2\pi) - \pi = 12 + 3\pi$.

Note that a perturbation as in Lemma \ref{lemcomp} to each of the four corners of the diagrams does not lead to a violation of the separation bound. In fact, there is a four-parameter family of minimal ribbon knot diagrams for each of these diagrams. It is easy to see that these perturbations leave the overall length of the knot constant.
\end{proof}

\begin{figure} [h!]
\centering
\includegraphics[width=.6\textwidth]{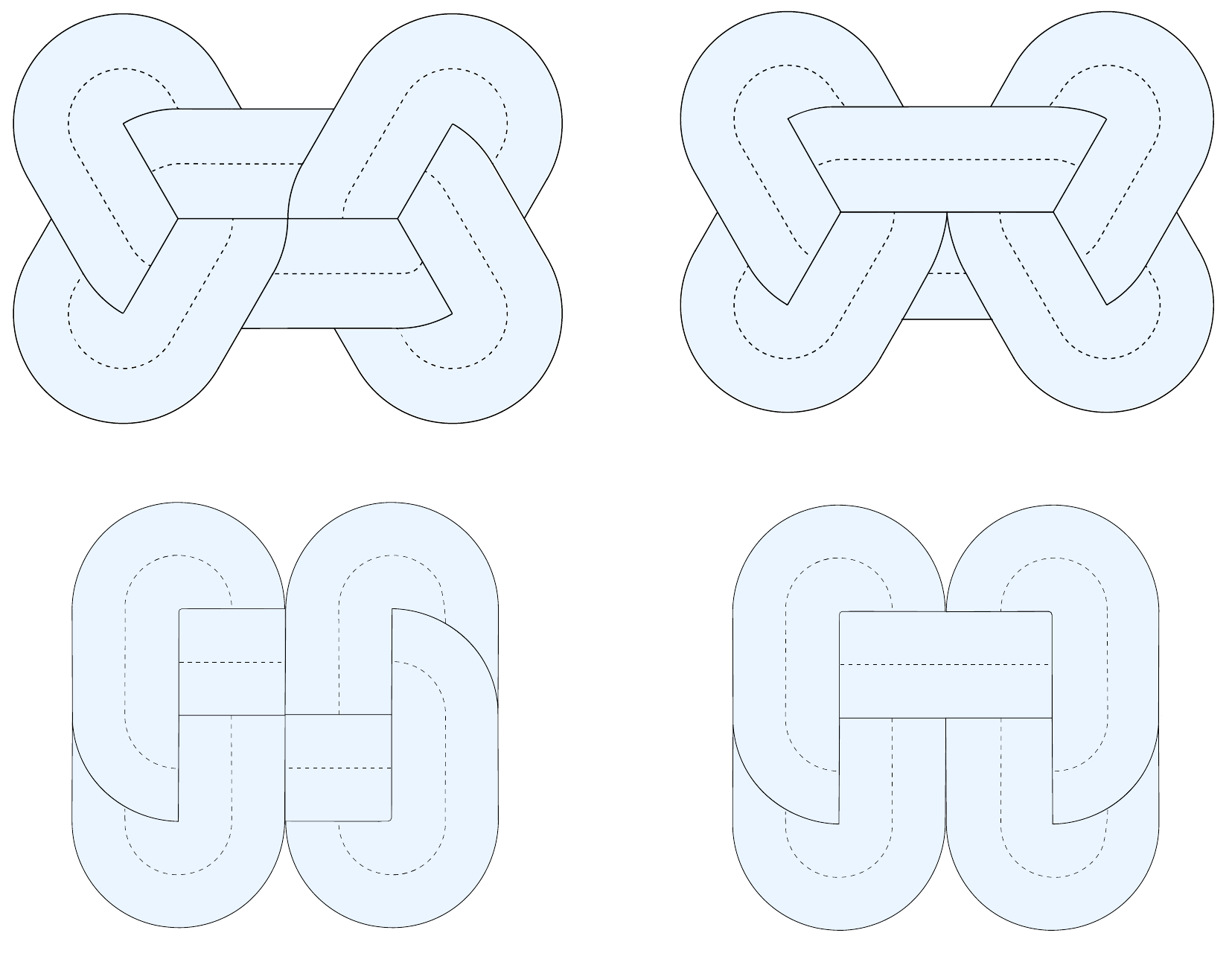}
\caption{Left: Minimal ribbon knot diagrams for the granny knot. Right: Minimal ribbon knot diagrams for the square knot. Note there is a four-parameter family of these objects. We can perturb the top diagrams to obtain the bottom ones and vice versa without changing their length. The second row shows tight configurations for the granny and square knots, respectively.}
\label{fig:}
\end{figure}

\section{Infinite Families of Ribbon Knots with Critical Ribbonlength} \label{infinite}

We propose a method for checking local criticality for certain classes of knots. This involves starting with a suitable {\it parallel} initial disk diagram so that a gradient descent ends in a configuration readily confirmed to be of {\it critical} ribbonlength. Therefore, any perturbation of its core can be verified to increase length. We consider the unit square lattice, along with Lemma \ref{lemcomp}, offering a simple setup for checking the presence of critical (not necessarily minimal) ribbon knot diagrams.

{\bf The Rectangular Family}. We examine $m \times n$ arrangements of disk diagrams by first proving their minimality in disk space and then confirming whether they satisfy the separation bound. By Theorem \ref{ribcs}, we can restrict our consideration to $cs$ curves. To construct these arrangements, we use $m$ sufficiently long parallel horizontal stadium curves. Next, we attach $n$ parallel, sufficiently spaced vertical stadium curves in an alternating crossing pattern, as shown in Figure \ref{fig:square}. Each resulting region accommodates a unit-radius disk with its center positioned within the unit square lattice.

\begin{theorem} \label{rect} A $m\times n$ arrangement in the rectangular family as minimal length representative with ribbonlength $\pi(m+n)+8mn$.
\end{theorem} 

\begin{proof} 
Consider a $cs$ diagram of an $m \times n$ rectangular arrangement. We apply the gradient descent method, consisting of vertical and horizontal linear perturbations (as in Theorem \ref{turk}), to maximise disk boundary contact while keeping the centres aligned with the unit square lattice. This process continues until the method converges to a local minimum lattice configuration.

The final configuration achieves all lower bounds, as each of the components are minimal perimeter stadium curves enclosing a finite number of disks. We conclude that the final $m \times n$ arrangement is length minimal. In Figure \ref{fig:square} (centre), we show a final-state disk diagram for a $2 \times 2$ arrangement.

Note that each horizontal component adds 8 units to the length of a vertical component. Conversely, each vertical component adds 8 units to the length of a horizontal component. Since we have an $m \times n$ rectangular arrangement, its core length is $2\pi(m+n) + 16mn$ and satisfies the separation bound. Therefore, its ribbonlength is $\pi(m+n) + 8mn$.
\end{proof}

\begin{figure} [h!]
\centering
\includegraphics[width=.9\textwidth]{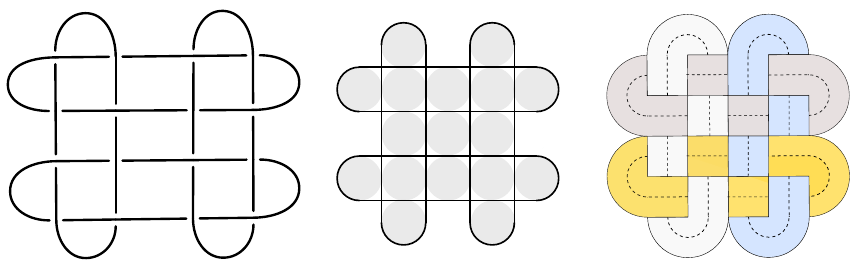}
\caption{Left: A rectangular disk diagram in a $2 \times 2$ arrangement, also known as the Salomon quadruple. Centre: A tightened disk diagram for the $2 \times 2$ arrangement. Right: A minimal ribbonlength $2 \times 2$ arrangement.}
\label{fig:square}
\end{figure}


\begin{observation} \label{altern} 
If we reverse the alternating pattern in the diagram in Figure \ref{fig:square}, the ribbonlength does not change. Moreover, in many cases, we could even make this diagram non-alternating and still have the same ribbonlength, however not minimal. This means that this realisation is critical for many different links with different global minima. Therefore, with this argument, we can establish rough upper bounds for several classes of ribbon knots.
\end{observation}

{\bf{The Salomon Family.}} The first element in this family is the Salomon knot (see Figure \ref{fig:sal1} above). The second element consists of a Salomon knot intertwined with an alternating unknotted loop (see Figure \ref{fig:sal1} below). The third element in this family is shown in Figure \ref{fig:door}, and so on.

So far we have established minimality since each of its constituent pieces is of minimal length, such as stadium curves. However, the general problem of determining the shortest curve that encloses $n$ unit disks in the plane is subtle and was investigated by Sch{\"u}rmann \cite{schurmann1}. In general, the minimum perimeter is not achieved by a subset of the hexagonal circle packing for $n > 370$. Whether this holds for $n < 54$ remains an open problem, which is why we cannot claim minimality in the following result.

\begin{figure} [h!]
\centering
\includegraphics[width=1\textwidth]{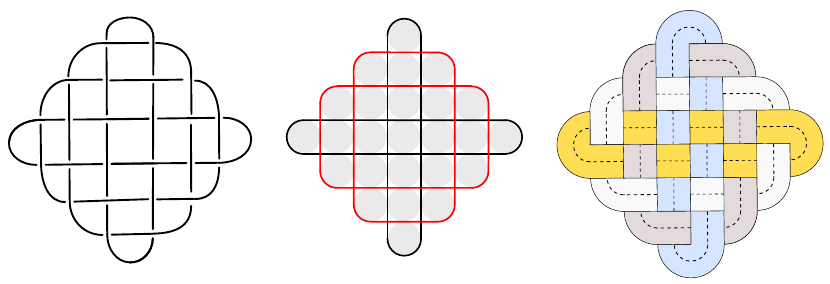}
\caption{Left: A diagram of the third element in the doormat family. Centre: A tightened disk diagram for the third element in the doormat family. Right: A critical ribbon diagram for the third element in the doormat family; minimality cannot be guaranteed due to \cite{schurmann1}, as for large values of $n$ there is no control over the positions of the centres of the disks in the diagram.}
\label{fig:door}
\end{figure}

\begin{theorem} \label{door} 
The $n$th element in the Salomon family is a ribbonlength critical diagram with ribbonlength
$$8 \sum_{k=1}^{n} k + (n+1)\pi.$$
\end{theorem}

\begin{proof} 

We emphasise that we are claiming {ribbonlength criticality}, not global minimality. Each component of the diagram is locally length minimising under the separation constraint (see Figure \ref{fig:door}). Similar to Theorem \ref{rect}, by starting with a suitable $cs$ disk diagram, we can reduce it to a disk diagram in lattice-tight form by applying linear perturbations. Hence, the diagram is critical with respect to the ribbonlength functional by Lemma \ref{lemcomp}. 

The first element in this family is the Salomon knot, which consists of two components with a total length of $(4 \cdot 4) + 2 \cdot 2\pi = 16 + 4\pi$ (see Theorem \ref{sal1}). 

The length of the second element is calculated by adding $4 \cdot 8$ units for the rectilinear contribution of a new loop, plus an additional $2\pi$ from the corners (each corner adding $\pi/2$), giving a total length of $(4 \cdot 4) + (4 \cdot 8) + 3 \cdot 2\pi = 48 + 6\pi$.

Following the same pattern, the length of the third element is:
$(4 \cdot 4) + (4 \cdot 8) + (4 \cdot 12) + 4 \cdot 2\pi = 96 + 8\pi$. In general, the formula for the length of the $n$th element is:
$$4 \cdot 4 \sum_{k=1}^{n} k + 2\pi(n+1) = 16 \sum_{k=1}^{n} k + 2\pi(n+1).$$
Thus, the ribbonlength for the $n$th element is:
$$8 \sum_{k=1}^{n} k + (n+1)\pi.$$
\end{proof}
\section{Open Problems and Conjectures}
\label{sec:conjectures}

The following presents a 2-dimensional analogue of the central problem in geometric knot theory:

\begin{problem}\label{prob:tabulation}
Provide a tabulation of knots ordered by minimal ribbonlength.
\end{problem}

A common assumption is that minimal ribbonlength configurations occur in diagrams with minimal crossing number. However, this assumption requires careful examination:

\begin{conjecture}\label{conj:linear-bound}
The minimal ribbonlength $\Rib(K)$ of a ribbon knot $K$ is bounded above by a linear function of its crossing number $c(K)$:

\[ \Rib(K) \leq C\cdot c(K) \quad \text{for some constant } C>0. \]
\end{conjecture}

This would establish an efficient geometric bound purely in terms of a topological invariant.

\begin{problem}\label{prob:relation-rib-rope}
What is the relationship between ribbonlength and ropelength?

While both functionals measure length under a fixed width or thickness constraint, they apply in different ambient settings, planar versus spatial. For example, the ribbonlength of the trefoil knot is exactly $\Rib(\text{trefoil}) = 6 + 2\pi \approx 12.28$, while numerical simulations estimate its minimal ropelength to be $\Rop(\text{trefoil}) \approx 16.37$. Is there a unifying principle that connects these two quantities, or do they reflect fundamentally different geometric contexts? Clarifying this relation may lead to a deeper understanding of energy minimisation in knot theory.
\end{problem}

In broad terms, it would involve a long term, ambitious plan to explore whether these geometric
constraints could offer new insights into problems in classical knot theory.

{}

\end{document}